\documentclass{amsart}
\usepackage{amsfonts}
\usepackage{amsmath}
\usepackage{graphicx}
 \usepackage{algorithmic} 
\usepackage{algorithm}
\usepackage{rotating}
\usepackage{hyperref}

\floatname{algorithm}{Algorithm}

\newtheorem{theorem}{Theorem}[section]
\newtheorem{lemma}[theorem]{Lemma}

\newtheorem{definition}[theorem]{Definition}
\newtheorem{lemma-definition}[theorem]{Lemma-Definition}

\newtheorem{corollary}[theorem]{Corollary}
\newtheorem{proposition}[theorem]{Proposition}
\newtheorem{lem-def}[theorem]{Lemma-Definition}

\newcommand{\R}{\mathbb R}

\newcommand{\Z}{\mathbb Z}
\newcommand{\Q}{\mathbb Q}

\newcommand{\F}{\mathbb F}

\def\op{\operatorname}

\def\as#1{\renewcommand\arraystretch{#1}}

\def\bb{{\mathcal B}}

\def\sp\operatorname{Spec}

\def\dsc{\mathrm{Disc}}
\def\diso{\lower.4ex\hbox{$\downarrow$}\raise.4ex\hbox{\mbox{\scriptsize
$\wr$}}}

\def\gen#1{\big\langle\, {#1} \,\big\rangle}
\def\gl#1#2{\op{GL}_{#1}(#2)}

\def\iso{\ \lower.3ex\hbox{\as{.08}$\begin{array}{c}\lra\\\mbox{\tiny $\sim\,$}\end{array}$}\ }

\def\Ks{K^{\op{sep}}}

\def\lg{l\raise.6ex\hbox to.2em{\hss.\hss}l}

\def\lra{\longrightarrow}
\def\m{{\mathfrak m}}

\def\oo{\mathcal{O}}
\def\orb{\hbox to  .3em{$\backslash$}\backslash}
\def\ord{\op{ord}}
\def\p{\mathfrak{p}}
\def\P{\mathfrak{P}}

\def\res{\operatorname{Res}}

\def\t{\theta}

\def\red{\operatorname{red}}
\def\bb{{\mathcal B}}

\def\gl{\mathrm{GL}}

\newcounter{cs}
\stepcounter{cs}
\newcommand{\fcasos}{\end{itemize}\setcounter{cs}{1}}

\newfont{\tit}{cmr12 scaled \magstep3}

\def\ww#1_#2{w_#2(#1)}
\def\WW#1_#2{w_#2\Big(#1\Big)}

\def\Kp{K_\p}

\def\kp{\hat{A}_\p}

\def\red{\operatorname{red}}
\def\sred{\operatorname{sred}}

\def\o1{{\mathcal O}_F}
\def\Ks{K^\mathrm{sep}}
\def\ti{{\bf{t}}}
\def\slp{\operatorname{Slopes}}
\def\tr{\mathrm{Trunc}}

\title{Computation of Integral Bases }

\thanks{This research was supported by MTM2012-34611 and MTM2013-40680 from the
Spanish MEC and by the Netherlands Organisation for 
Scientific Research (NWO) under grant 613.001.011.}

\author{Jens-Dietrich Bauch}
\address{Department of Mathematics and Computer Science
Technische Universiteit Eindhoven
P.O. Box 513, 5600 MB Eindhoven, The Netherlands}
\email{j.bauch@tue.nl}

\makeatletter

\begin{document}

\begin{abstract}
Let $A$ be a Dedekind domain, $K$ the fraction field of $A$, and $f\in A[x]$ a monic irreducible separable polynomial. For a given non-zero prime ideal $\p$ of $A$ we present in this paper a new characterization of a $\p$-integral basis of the extension of $K$ determined by $f$. This characterization yields in an algorithm to compute $\p$-integral bases, which is based on the use of simple multipliers that can be constructed with the data that occurs along the flow of the Montes Algorithm. Our construction of a $\p$-integral basis is significantly faster than the similar approach from \cite{HN} and provides in many cases a priori a triangular basis.
\end{abstract}
\maketitle

\section*{Introduction}

Let $A$ be a Dedekind domain, $K$ the fraction field of $A$, and $\p$ a non-zero prime ideal of $A$. By $A_\p$ we denote the localization of $A$ at $\p$. Let $\pi\in\p$ be a prime element of $\p$.

Denote by $\t\in \Ks$ a root of a monic irreducible separable polynomial $f\in A[x]$ of degree $n$ and let $L=K(\t)$ be the finite separable extension of $K$ generated by $\t$. We denote by $\oo$ the integral closure of $A$ in $L$ and by $\oo_\p$ the integral closure of $A_\p$ in $L$. A $\p$-integral basis of $\oo$ is an $A_\p$-basis of $\oo_\p$ (cf. Definition \ref{DefpInt}).

If $A$ is a PID, then $\oo$ is a free $A$-module of rank $n$, and its easy to construct an $A$-basis of $\oo$ from the different $\p$-integral bases, for prime ideals $\p$ of $A$ that divide the discriminant of $f$.

In this work we follow the approach from \cite{HN} to apply the notion of reduceness in the context of integral bases. By weakening the concept of reduceness we deduce a new characterization of $\p$-integral bases (Theorem \ref{newdesofint}). This yields in an algorithm to compute a $\p$-integral basis: We construct for any prime ideal $\P$ of $\oo$ lying over $\p$ a local set $\bb^*_\P\subset \oo$ and a multiplier $z_\P\in L$ such that $\cup_{\P|\p}z_\P\bb^*_\P$ is a $\p$-integral basis of $\oo$, where $z_\P\bb^*_\P$ denotes the set we obtain by multiplying all elements in $\bb_\P^*$ by $z_\P$. The construction of these local sets and the multipliers is based on the \textit{Okutsu-Montes (OM) representations} of the prime ideals of $\oo$ lying over $\p$, provided by the Montes algorithm. In comparison with the existing methods from \cite{NewC} and \cite{HN} our construction of the multipliers is much simpler (and faster) and results in many cases directly in a triangular $\p$-integral basis $\bb$ of $\oo$, that is, $\bb=\{b_0,\dots,b_{n-1}\}$, where $b_{i}=g_{i}(\t)/\pi^{m_i}$ with $g_{i}\in A[x]$, monic of degree $i$ and $m_i\in\Z$. Hence the transformation into a basis in HNF becomes especially efficient.

The article is divided in the following sections. In section \ref{Sec1} we summarize the Montes algorithm briefly and introduce the basic ingredients of our algorithm for the computation of a $\p$-integral basis. That is, we define types, Okutsu invariants, and a local set $\bb_\P\subset A[x]$ (cf. Definition \ref{divpolset}) for a prime ideal $\P$ of $\oo$ lying over $\p$. In section \ref{Sec2} we introduce the notion of (semi)-reduced bases, which provides a new characterization of $\p$-integral bases (Theorem \ref{newdesofint}) and a new method of constructing multipliers $z_\P$, for any prime ideal $\P$ of $\oo$ over $\p$, such that the union of the sets $\{z_\P\cdot b(\t)/\pi^{m_b}\mid b\in \bb_\P\}$, for $\P|\p$ and certain integers $m_b$, is a $\p$-integral basis. If we assume that $A/\p$ is finite with $q$ elements and $\mathcal{R}$ is a set of representatives of $A/\p$ then we will see that the complexity of the method is dominated by $O\left(n^{1+\epsilon}\delta\log q+n^{1+\epsilon}\delta^{2+\epsilon}+n^{2+\epsilon}\delta^{1+\epsilon}\right)$ operation in $\mathcal{R}$ (Lemma \ref{compli}), where $\delta := v_\p(\dsc f)$. In section \ref{Sec4} we consider the practical performance of our method in the context of algebraic function fields. We have implemented the method for the case $A=k[t]$, where $k$ is a finite field or $k=\Q$. The package can be downloaded from \url{https://github.com/JensBauch/Integral_Basis}. 

\section{Montes algorithm}\label{Sec1}

We consider the monic separable and irreducible polynomial $f\in A[x]$. For a non-zero prime ideal $\p$ of $A$ we denote the induced discrete valuation by $v_\p: A\rightarrow \Z\cup \{\infty\}$ and the completion of $K$ at $\p$ by $K_\p$. The valuation $v_\p$ extends in an obvious way to $K_\p$. Denote by $\hat{A}_\p$ the valuation ring of $v_\p$ and by $\m_\p=\p \hat{A}_\p$ its maximal ideal.

By the classical theorem of Hensel \cite{Hensel} the prime ideals of $\oo$ lying over $\p$ are in one-to-one correspondence with the monic irreducible factors of $f$ in $\hat{A}_\p[x]$. 

In this section we describe the Montes algorithm, which determines for the input of $f$ and $\pi$ a parametrization of the irreducible factors of $f$ in $\hat{A}_\p[x]$. Let $\P$ be a prime ideal of $\oo$ lying over $\p$ and denote by $f_\P\in \hat{A}_\p[x]$ the corresponding irreducible factor of $f$. Then, the Montes algorithm produces a list of data, a so-called \emph{type},
$$
\ti=(\psi_0;(\phi_1,\lambda_1,\psi_1);\dots ;(\phi_{r+1},\lambda_{r+1},\psi_{r+1})),
$$
which is a representation of the irreducible factor $f_\P$ and therefore a representation the prime ideal $\P$. We call this representation an OM-representation of $\P$ (cf. Definition \ref{omrep}). 

The Montes algorithm can be seen as a factorization algorithm, which detects the factorization of $f\in \hat{A}_\p[x]$, but never computes it. To this purpose, a kind of Hensel's lemma of higher order is applied \cite[Theorem 3.7]{HN}. At any level $i\geq 1$, besides the fundamental data $\phi_i\in A[x],\ \lambda_i\in\Q_{>0},\ \psi_i\in\F_i[y]$, where $\F_i$ is a finite extension of $k_\p:=A/\p$, the type supports:
\begin{itemize}
\item $N_i:\hat{A}_\p[x]\rightarrow 2^{\R^2}$ a \textbf{Newton polygon operator},
\item $R_i:\hat{A}_\p[x]\rightarrow \F_i[y]$ a \textbf{residual polynomial operator},
\item $v_{i-1}:K_\p(x)\rightarrow \Z\cup\{\infty\}$ a \textbf{discrete valuation}.
 
\end{itemize}
Below we give a brief overview of the Montes algorithm, the OM-representation of prime ideals, and certain applications, which will be useful for further considerations. The results are mainly extracted from \cite{HN1} and \cite{NewC}. A comprehensive explanation of the Montes algorithm can be found in \cite{HN2}.

\subsection{Types}\label{omreppi}
We consider $v_\p:K_\p\rightarrow \Q\cup\{\infty\}$ the by $\p$ induced valuation on $K_\p$ and extend it to a discrete valuation $v_0$ on $K_\p(x)$, determined by
\begin{align}\label{v0}
v_0:K_\p[x]\rightarrow \Z\cup\{\infty\},\quad v_0(c_0+\cdots +c_rx^r):=\min\{v_\p(c_i)\mid 0\leq j\leq r\}.
\end{align}
\underline{\emph{Types of order zero}}\medskip\\
We denote by $\F_0:=k_\p=A/\p$ and define the $0$-th residual polynomial operator
$$
R_0:\hat{A}_\p[x]\rightarrow \F_0[y],\quad g(x)\mapsto \overline{g(y)/\pi^{v_0(g)}},
$$
where $\overline{\phantom{ii}}:\hat{A}_\p[y]\rightarrow \F_0[y]$ is the natural reduction map and $\pi\in\p$ a uniformizer. A \emph{type of order zero},
$$
\ti=(\psi_0),
$$
is determined by $\psi_0(y)\in\F_0[y]$, a monic irreducible polynomial. A \emph{representative} of $\ti$ is any monic polynomial $\phi_1(x)\in A[x]$ such that $R_0(\phi_1)=\psi_0$. 

We consider the polynomial $f$ in $A[x]$. From a factorization of $R_0(f)(y)=\psi_{1,0}^{n_1}\cdots \psi_{\kappa,0}^{n_\kappa}$ into the product of irreducible monic polynomials $\psi_{i,0}(y)\in\F_0[y]$ we deduce types of order zero. Each irreducible factor $\psi_{i,0}(y)$ singles out one type of order zero. For convenience, we consider one fixed factor, denote it by $\psi_0$, and consider a representative $\phi_1\in A[x]$. Let $m_1:=\deg \phi_1$.\medskip\\
\underline{\emph{Types of order one}:}\\

\textbf{Newton polygon operator.} The Newton polygon of a polynomial $g(x)\in K_\p[x]$ is determined by the pair $(v_0,\phi_1)$. If $\sum_{s\geq 0} a_s(x)\phi_1(x)^s$ is the $\phi_1$-adic development of $g(x)$, then 
\begin{align}\label{NP}
N_1(g):=N_{v_0,\phi_1}(g)
\end{align}
is defined to be the lower convex hull of the set of points of the plane with coordinates $(s,v_0(a_s(x)\phi_1(x)^s))$. However, we only consider the principal part of this polygon, $N^{-}_1(g)=N^{-}_{v_0,\phi_1}(g)$, formed by the sides of negative slopes of $N_1(g)$. The length $l(N^{-}_1(g))$ of the polygon $N^{-}_1(g)$ is, by definition, the abscissa of its right end point. The typical shape of $N^{-}_1(g)$ for a monic polynomial $g$ is as shown below.

\begin{figure}[!h]\caption{Newton polygon of $g$.}
\label{fig1}
\begin{center}
\setlength{\unitlength}{5mm}
\begin{picture}(12,5)
\put(1.3,3.89){$\bullet$}
\put(1.82,2.84){$\bullet$}
\put(2.30,1.86){$\bullet$}
\put(3.30,1.86){$\bullet$}
\put(6.30,0.8){$\bullet$}
\put(5.30,-0.2){$\bullet$}
\put(8.30,-0.2){$\bullet$}
\put(10.30,-0.2){$\bullet$}

\put(-1,0){\line(1,0){12}}\put(0,-1){\line(0,1){6}}
\put(2.42,2.2){\line(-1,2){1}}\put(2.44,2.2){\line(-1,2){1}}
\put(5.50,0.0){\line(-3,2){3}}\put(5.52,0.0){\line(-3,2){3}}
\put(5.5,0.02){\line(1,0){5}}
\put(-0.5,-0.7){\begin{footnotesize}$0$\end{footnotesize}}
\put(2.5,2.9){\begin{footnotesize}$N^{-}_1(g)$\end{footnotesize}}
\put(0.7,-0.7){\begin{footnotesize}$\ord_{\phi_1}(g)$\end{footnotesize}}
\put(4.7,-0.7){\begin{footnotesize}$l(N^{-}_1(g))$\end{footnotesize}}
\put(9.1,-0.7){\begin{footnotesize}$\lfloor \deg g/m_1\rfloor$\end{footnotesize}}

\multiput(1.45,-0.03)(0,0.4){11}
{\line(0,1){0.2}}
\end{picture}
\end{center}
\end{figure}
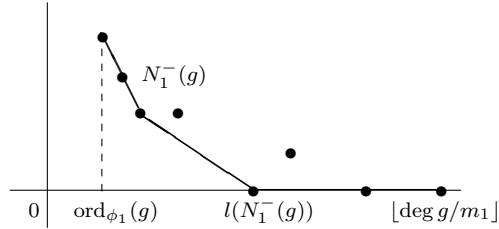
\bigskip
\textbf{Residual polynomial operator.} We fix $\F_1:=\F_0(y)/(\psi_0(y))$ and we set $z_0$ to be the class of $y$ in $\F_1$, so that $\F_1=\F_0[z_0]$. The polygon $N:=N^{-}_1(g)$ has a residual coefficient $c_s$ at each integer abscissa, $\ord_{\phi_1}g\leq s\leq l(N)$, defined as follows:
$$
c_s:=\begin{cases}
0, &\text{if }(s,v_0(a_s)) \text{ lies above }N,\\ 
R_0(a_s)(z_0)\in \F_1, &\text{if }(s,v_0(a_s)) \text{ lies on }N. 
\end{cases}
$$

\noindent Denote by $\slp(N)$ the set of slopes of $N$. Given any $\lambda\in\Q_{>0}$, we define:
$$
S_{\lambda}(N):=\{(x,y)\in N\mid y+{\lambda}x \text{ is minimal}\}=\begin{cases}\text{ a vertex,} &\text{ if }{-\lambda}\notin \slp(N),\\\text{ a side,}&\text{ if }{-\lambda}\in \slp(N).
\end{cases}
$$
The picture below illustrates both possibilities. In this picture $L_{{\lambda}}$ is the line of slope $-\lambda$ having first contact with $N$ from below.

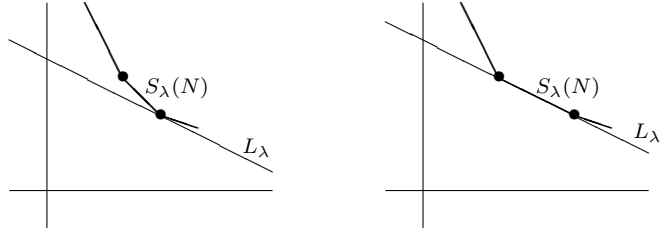
\begin{figure}[h]\caption{${\lambda}$-component of a polygon.}\label{figComponent}
\begin{center}
\setlength{\unitlength}{5mm}
\begin{picture}(15,6)
\put(2.85,1.85){$\bullet$}\put(1.85,2.85){$\bullet$}
\put(-1,0){\line(1,0){7}}\put(0,-1){\line(0,1){6}}
\put(3,2){\line(-1,1){1}}\put(3.02,2){\line(-1,1){1}}
\put(3,2){\line(3,-1){1}}\put(3.02,2){\line(3,-1){1}}
\put(2,3){\line(-1,2){1}}\put(2.02,3){\line(-1,2){1}}
\put(6,.5){\line(-2,1){7}}\put(5.2,1){\begin{footnotesize}$L_{{\lambda}}$\end{footnotesize}}
\put(13.85,1.85){$\bullet$}\put(11.85,2.85){$\bullet$}
\put(9,0){\line(1,0){7}}\put(10,-1){\line(0,1){6}}
\put(14,2){\line(-2,1){2}}\put(14.02,2){\line(-2,1){2}}
\put(14,2){\line(3,-1){1}}\put(14.02,2){\line(3,-1){1}}
\put(12,3){\line(-1,2){1}}\put(12.02,3){\line(-1,2){1}}
\put(16,1){\line(-2,1){7}}\put(15.6,1.4){\begin{footnotesize}$L_{{\lambda}}$\end{footnotesize}}
\put(13,2.6){\begin{footnotesize}$S_{\lambda}(N)$\end{footnotesize}}
\put(2.6,2.6){\begin{footnotesize}$S_{\lambda}(N)$\end{footnotesize}}
\end{picture}
\end{center}
\end{figure}

In any case, $S_{\lambda}(N)$ is a segment of $\R^2$ with end points having integer coordinates. Any such segment has a degree. If ${\lambda} = h/e$ with $h,e$ positive coprime integers, the degree of $S_{\lambda}(N)$ is defined as:
$$
d := d(S_{\lambda}(N)) := l(S_{\lambda}(N))/e,
$$
where $l(S_{\lambda}(N))$ is the length of the projection of $S_{\lambda}(N)$ to the horizontal axis. Note that $S_{\lambda}$ splits into $d$ minimal subsegments, whose end points have integer coordinates. Denote $s_0$ and $s_1$ the abscissas of the endpoints of $S_{\lambda}$. Then, the abscissas of the points on $S_{\lambda}$ with integer coordinates are given by $s_0,s_0+e,\dots,s_1=s_0+de$.
We define the residual polynomial of first order of $f(x)$, with respect to $v_0,\phi_1,{\lambda}$, as:
$$
R_{v_0,\phi_1,{\lambda}}(g)(y):=c_{s_0}+c_{s_0+e}y+\cdots+c_{s_1}y^d\in\F_1[y].
$$
Note that $c_{s_0}c_{s_1}\neq 0$; thus, the degree of $R_{v_0,\phi_1,{\lambda}}(g)$ is always equal to $d$. Let $h_1,e_1$ be coprime positive integers and consider the positive rational number $\lambda_1=h_1/e_1$. Let $\psi_1(y)\in \F_1[y]$ be a monic irreducible polynomial with $\psi_1(y)\neq y$. Then,
$$
\ti=(\psi_0;(\phi_1,\lambda_1,\psi_1)),
$$
is called a \emph{type of order one}. Such a type supports a \emph{residual polynomial operator} of the first order $R_1:=R_{v_0,\phi_1,{\lambda_1}}$. Given any such type, one can compute a \emph{representative} of $\ti$; that is, any monic polynomial $\phi_2(x)\in A[x]$ of degree $e_1 \deg \psi_1 \deg \phi_1$, satisfying $R_1(\phi_2)(y)=\psi_1(y)$. This polynomial is necessarily irreducible in $\hat{A}_\p[x]$.
\bigskip

\textbf{Discrete valuation.} The triple $(v_0,\phi_1,{\lambda_1})$ also determines a discrete valuation on $K_\p(x)$ as follows: For $g\in K_\p[x]$ nonzero, we consider the intersection point $(0,H)$ of the vertical axis with the line of slope $-\lambda_1$ containing $S_{\lambda_1}(N^{-}_1(g))$. Then, we set $v_1(g(x)):=e_1H$.\medskip\\
\underline{\emph{Types of order} $r$:}\medskip\\
Now we may start over again with the pair $(v_1,\phi_2)$ and repeat all constructions in order two. The iteration of this procedure leads to the concept of a \emph{type of order} $r$.

A type of order $r\geq 1$ is a chain:
$$
\ti=(\psi_0;(\phi_1,\lambda_1,\psi_1);\dots;(\phi_r,\lambda_r,\psi_r)),
$$ 
where $\phi_1(x),\dots,\phi_r(x)\in A[x]$ are monic and irreducible in $\hat{A}_\p[x]$, $\lambda_1,\dots,\lambda_r\in \Q_{>0}$, and $\psi_0(y)\in\F_0[y],\dots,\psi_r(y)\in\F_r[y]$ are monic irreducible polynomials over certain fields $\F_0\subset \dots\subset \F_r$ that satisfy the following recursive properties:
\begin{enumerate}
\item $R_0(\phi_1)(y)=\psi_0(y)$. We define $\F_1:=\F_0(y)/(\psi_0(y))$.
\item For all $1\leq i<r$,
\begin{itemize}
\item	$\deg \phi_i|\deg \phi_{i+1}$,
\item $N_i{(\phi_{i+1})}:=N_{v_{i-1},\phi_i}(\phi_{i+1})$ is one-sided of slope $-\lambda_i$, and
\item $R_i(\phi_{i+1})(y):=R_{v_{i-1},\phi_i,\lambda_i}(\phi_{i+1})(y)=\psi_i(y)$.
\end{itemize}
We define $\F_{i+1}:=\F_{i}[y]/(\psi_i(y))$.

\item $\psi_r(y)\neq y$. We define $\F_{r+1}:=\F_{r}[y]/(\psi_r(y))$.
\end{enumerate}
Thus, a type of order $r$ is an object structured in r levels. In the computational representation of a type, several invariants are stored at each level, $1\leq i\leq r$. The most important ones are:
\begin{center}
\begin{tabular}{ll}
$\phi_i(x)$,			&monic polynomial in $A[x]$, irreducible in $\hat{A}_\p[x]$,\\
$m_i$,				&$\deg\phi_i(x)$,\\
$V_i:=v_{i-1}(\phi_i)$,		&nonnegative integer,\\
$\lambda_i=h_i/e_i$,	&$h_i,e_i$ positive coprime integers\\
$\psi_i(y)$,			&monic irreducible polynomial in $\F_i[y]$,\\
$f_i$,				&$\deg \psi_i(y)$,\\
$z_i$,				&the class of $y$ in $\F_{i+1}$, so that $\psi_i(z_i)=0$ and $\F_{i+1}=\F_i[z_i]$.

\end{tabular}
\end{center}

\noindent Take $f_0:=\deg \psi_0$. Note that 
\begin{align}\label{mifi}
m_i=(f_0f_1\cdots f_{i-1})(e_1\cdots e_{i-1})=e_{i-1}f_{i-1}m_{i-1},\quad \dim_{\F_0}\F_{i+1}=f_0f_1\cdots f_{i}.
\end{align}
The discrete valuations $v_0,\dots,v_{r}$ on the field $K_\p(x)$ are essential invariants of the type.
\begin{definition}\label{gcomplete}
Let $g(x)\in \hat{A}_\p[x]$ be a monic polynomial, and $\ti$ a type of order $r\geq 1$.
\begin{enumerate}
\item We say that $\ti$ \emph{divides} $g(x)$, if $\psi_r(y)$ divides $R_r(g)(y)$ in $\F_r[y]$. We denote $\ord_\ti(g):=\ord_{\psi_r}(R_r(g))$
\item We say that $\ti$ is $g$-\emph{complete} if $\ord_\ti(g)=1$. In this case, $\ti$ singles out a monic irreducible factor $g_{\ti}(x)\in \hat{A}_\p[x]$ of $g(x)$, uniquely determined by the property $R_{r}(g_{\ti})(y)=\psi_r(y)$. If $K_\ti$ is the extension field of $K_\p$ determined by $g_{\ti}(x)$, then
$$
e(K_{\ti}/K_\p)=e_1\cdots e_r,\quad f(K_{\ti}/K_\p)=f_0f_1\cdots f_r. \label{ramindex}
$$
\item A \emph{representative} of $\ti$ is a monic polynomial $\phi_{r+1}(x)\in A[x]$, of degree $m_{r+1}=e_rf_rm_r$ such that $R_{r}(\phi_{r+1})(y)=\psi_r(y)$. This polynomial is necessarily irreducible in $\hat{A}_\p[x]$. By definition of a type, each $\phi_{i+1}$ is a representative of the truncated type of order $i$
$$
\tr_i(\ti):=(\psi_0;(\phi_1,\lambda_1,\psi_1);\dots;(\phi_i,\lambda_i,\psi_i)).
$$

\item We say that $\ti$ is optimal if $m_1<\cdots<m_r$, or equivalently, if $e_if_i>1$, for all $1\leq i<r$.
\end{enumerate}
A type $\ti$ of order $0$ is by definition optimal. 
\end{definition}

\subsection{The Montes algorithm}
For given $f(x)$ and an uniformizer $\pi$ of $\p$, the Montes algorithm determines a family $\ti_1,\dots,\ti_\kappa$ of $f$-complete and optimal types, which correspond uniquely to the irreducible factors of $f_{\ti_1},\dots,f_{\ti_\kappa}$ of $f$ in $\hat{A}_\p[x]$. This correspondence is determined by
\begin{align*}
\ord_{\ti_i}(f_{\ti_j})=\begin{cases}
0, &\text{ if }i\neq j,\\
1, &\text{ if }i= j.
\end{cases}
\end{align*}

Initially the algorithm computes the order zero types determined by the irreducible factors of $f(x)$ modulo $\p$, and then enlarges them successively in an adequate way until the whole list of $f$-complete optimal types is obtained. 

Every output $\ti$ of the Montes algorithm is a type of order $r+1$, where $r$ is called the \emph{Okutsu depth} of the corresponding irreducible factor $f_\ti(x)$. The sequence $[\phi_1,\dots,\phi_{r}]$ is an \emph{Okutsu frame} of $f_\ti(x)$. Details can be found in \cite{OIN}.

The invariants $v_{i},h_i,e_i,f_i$ at each level $0\leq i\leq r$ are canonical (depend only on $f(x)$). The $(r+1)$-level $\ti$ carries only the invariants:
$$
\phi_{r+1},m_{r+1},V_{r+1},\lambda_{r+1}=-h_{r+1},e_{r+1}=1,\psi_{r+1},f_{r+1}=1.
$$
If $\P$ is the prime ideal corresponding to $\ti$, we denote 
\begin{align*}
&f_\P(x):=f_\ti(x)\in \hat{A}_\p[x],\quad \phi_\P(x):=\phi_{r+1}(x)\in A[x],\\
&n_\P:=m_{r+1}=\deg f_\P=\deg \phi_\P,\\
&\ti_\P:=\ti=(\psi_0;(\phi_1,\lambda_1,\psi_1);\dots;(\phi_r,\lambda_r,\psi_r);(\phi_\P,\lambda_{r+1},\psi_{r+1})).
\end{align*}
\begin{definition}\label{omrep}
We say that $\ti_\P$ is an OM \emph{representation} of $\P$.
\end{definition}

\begin{algorithm}
\caption{: Montes algorithm}
\label{AlgoMontes}
\begin{algorithmic}[1]

\REQUIRE A monic irreducible separable polynomial $f\in A[x]$ and a uniformizer $\pi$ of a prime ideal $\p$ of $A$.
\ENSURE A family $\ti_1,\dots,\ti_s$ of $f$-complete and optimal types, parameterizing the monic irreducible factors $f_{\P_1}(x),\dots,f_{\P_s}(x)$ of $f$ in $\hat{A}_\p[x]$.\\[0.25 cm]

\end{algorithmic}
\end{algorithm}

\subsection{Okutsu approximations}\label{OkApp}

Denote $\P$ a prime ideal of $\oo$ over $\p$ and let
$$
\ti_\P=(\psi_0;(\phi_1,\lambda_1,\psi_1);\dots;(\phi_r,\lambda_r,\psi_r);(\phi_\P,\lambda_{r+1},\psi_{r+1}))
$$
be an OM representation of $\p$. The polynomial $\phi_\P(x)$ is an \emph{Okutsu approximation} to the irreducible factor $f_\P(x):=f_{\ti_\P}(x)$ \cite[Sect. 4.1]{OIN}. The value $\lambda_{r+1}=h_{r+1}$ is not a canonical invariant of $f_{\P}$. It measures how close is $\phi_\P$ to $f_\P$; we have $\phi_\P=f_{\P}$ if and only if $h_{r+1}=\infty$. 

In order to determine a $\p$-integral basis our algorithm (cf. Algorithm \ref{AlgopINT}) requires the computation of an Okutsu approximation $\phi_\P$ with sufficiently large value $\lambda_{r+1}=h_{r+1}$. This can be achieved by applying the \emph{single-factor lifting} algorithm of \cite{GNP}, which improves the Okutsu approximation to $f_\P$ with quadratic convergence; that is, doubling the value of $h_{r+1}$ at each iteration. By \cite[p. 744]{NewC} it holds
$$
v_\P(\phi_\P(\t))=V_{r+1}+\lambda_{r+1},
$$
where $V_{r+1}=e_rf_r(e_rV_r+h_r)$ (cf. \cite[p. 141]{BNS}) is an invariant of the type $\ti_\P$. By \cite[Proposition 4.7]{NewC} the value $v_{\mathfrak{P}'}(\phi_\P(\t))$ is given by a closed formula in terms of the data attached to the types $\ti_\P,\ \ti_{\mathfrak{P}'}$, for any prime ideal $\mathfrak{P}'$ of $\oo$ lying over $\p$ different from $\P$. Hence, the single-factor lifting algorithm can produce an element $\phi_\P(\t)$ in $L$ with arbitrary large valuation at $\P$ and constant value $v_{\mathfrak{P}'}(\phi_\P(\t))$, for $\mathfrak{P}'|p$ with $\mathfrak{P}'\neq \P$.

\begin{algorithm}\label{Algo2}
\caption{: Single-factor lifting}\label{SFL}
\begin{algorithmic}[1]

\REQUIRE An OM representation $\ti_\P$ of a prime ideal $\P$ with an Okutsu approximation $\phi_\P$ and $h\in \Z$.
\ENSURE An Okutsu approximation $\phi'_\P$ with $v_\P(\phi'_\P(\t))\geq V_{r+1}+h$.\\[0.25 cm]

\end{algorithmic}
\end{algorithm}

\subsection{Divisor polynomials}\label{divPol}

The notion of \emph{divisor polynomials} is due to Okutsu \cite{Oku}. These polynomials will play a fundamental role in the context of the computation of a $\p$-integral bases. The results are extracted from \cite{OIN}. A comprehensive explanation and proofs can be found there.

For any prime ideal $\P$ lying over $\p$, we consider the data $\ti_\P,\ \phi_\P$ obtained by the Montes algorithm. Additionally, we choose a root $\t_\P$ in $\overline{ K}_\p$ of $f_\P$ and consider the local field $L_\P:=K_\p(\t_\P)$. In particular, $L_\P$ is an extension of $K_\p$ of degree $n_\P=\deg f_\P$. 
As before, we denote by $v_\p$ the discrete valuation on $K$ induced by the prime ideal $\p$ of $A$ and denote by $\hat{v}$ its canonical extension to an algebraic closure of $K_\p$. Recall that $L=K(\t)$, where $\t$ is a root of $f(x)$. Consider the topological embedding $\iota_\P:L\hookrightarrow L_\P$, determined by $\t\mapsto \t_{\P}$. Let $v_0$ be defined as in (\ref{v0}).

\begin{proposition}\label{propdivpol}
For any integer $0\leq m<n_\P$, there exists a monic polynomial $g_m(x)\in \hat{A}_\p[x]$ of degree $m$ such that
$$
\hat{v}(g_m(\t_\P))\geq \hat{v}(g(\t_\P))-v_0(g(x)),
$$
for all polynomials $g(x)\in \hat{A}_\p[x]$ having degree $m$.
\end{proposition}
Note that the valuation condition from the last proposition does not depend on the choice of the root $\t_\P$ of $f_\P$.

\begin{definition}
We call $g_m(x)$ a \emph{divisor polynomial} of degree $m$ of $f_\P$.
\end{definition}

\begin{lemma}\label{divisorpolred}
Let $0\leq i<j<n_\P$ and $g_i(x), g_j(x)$ two divisor polynomials of $f_\P$. Then,
$$
\hat{v}(g_j(\t_\P))\geq \hat{v}(g_i(\t_\P)).
$$
\end{lemma}
\begin{proof}
Since $x^{j-i}g_i(x)\in \hat{A}_\p[x]$ is monic and has degree equal $j$, the last proposition shows that $
\hat{v}(g_j(\t_\P))\geq \hat{v}(\t_\P^{j-i}g_i(\t_\P))\geq \hat{v}(g_i(\t_\P)).
$
\end{proof}

Let $\ti_\P$ be an OM representation of the prime ideal $\P$ lying over $\p$, with $\phi$-polynomials $\phi_1,\dots,\phi_r$. We fix $\phi_0:=x$. Recall that $m_i=\deg \phi_i$, for $0\leq i\leq r$.

\begin{theorem}\label{divisorpoltheo}
For $0<m<n_\P$, we write uniquely 
$$
m=\sum_{i=0}^rc_im_i,\quad 0\leq c_i<\frac{m_{i+1}}{m_i}.
$$
Then, $g_m(x):=\prod_{i=0}^r\phi_i(x)^{c_i}$ is a divisor polynomial of degree $m$ of $f_\P$.
\end{theorem}

\begin{definition}\label{divpolset}
For a prime ideal $\P$ of $\oo$ we define 
$$
\bb_\P:=\{g_0(x),g_1(x),\dots, g_{n_\P-1}(x)\}
$$
\end{definition}
\noindent with $g_0(x)=1$.

Note that, for a prime ideal $\P$, the set $\bb_\P$ is a subset of $A[x]$. The set of all $g_m(\t_\P)/\pi^{\lfloor\hat{v}(g_m(\t_\P)) \rfloor}$, for $0\leq m<n_\P$, is an $\hat{A}_\p$-basis of the integral closure of $\hat{A}_\p$ in the finite extension $K_\p(\t_\P)$. This basis is called the \emph{Okutsu basis} of $\P$.

\section{Reduced bases}\label{Sec2}

In order to describe our algorithm for the computation of a $\p$-integral basis we use the notion of reduced bases which was introduced by W.M.Schmidt \cite{Sch} in the context of algebraic function fields.

By weakening the concept of reduceness we obtain the notion of semi-reduceness and deduce a new characterization of $\p$-integral bases (cf. Theorem \ref{newdesofint}). Recall that $L=K(\t)$ is the finite extension of $K$ determined by the monic separable and irreducible polynomial $f$.

Let $v_\p$ be the discrete valuation determined by the non-zero prime ideal $\p$ and $\pi\in \p$ a uniformizer. We define a $v_\p$-compatible norm.

\begin{definition}
A $\p$-\emph{norm} on $L$ is a mapping $w:L \rightarrow \Q\cup\{\infty\}$ that satisfies:
\begin{enumerate}
\item $w(x+y) \geq \min\{w(x),w(y)\}$, $\forall x,y\in L$, and equality holds if $w(x)\not=w(y)$,
\item $w(ax) =v_\p (a) +w(x) $, for all $a\in K$ and $x\in L$, and
\item $w(x) =\infty$ if and only if $x=0$.

\end{enumerate}
\end{definition}
In other words, a $\p$-norm is an extension of the valuation $v_\p $ to the finite extension $L$ having all properties of a valuation except for the good behavior with respect to multiplication. 

\begin{definition} \label{taured}
Let $w $ be a $\p$-norm. The set $\bb\subset L$ is called $w $-reduced if 
\begin{align}\label{defRed1}
w\Big(\sum_{b\in \bb}\lambda_b b\Big) =\min_{b\in \bb}\{w(\lambda_b b) \}
\end{align}
for all $\lambda_b\in K$. If we have additionally $0\leq w( b) < 1$, for all $b\in\bb$, then we call $\bb$ $w $-orthonormal.

If we weaken condition \textup{(\ref{defRed1})} to 
$$
\Big\lfloor w\Big(\sum_{b\in \bb}\lambda_b b\Big) \Big\rfloor=\min_{b\in \bb}\{\lfloor w(\lambda_b b) \rfloor\},
$$
we call $\bb$ $w $-semi-reduced or $w $-semi-orthonormal, respectively.
\end{definition}
Equivalently, $\bb$ is $w$-reduced if and only if (\ref{defRed1}) holds for coefficients $\lambda_b\in A$ not all of them divisible by $\p$.

If the $\p$-norm $w $ is fixed we just say (semi-) reduced or (semi-) orthonormal, respectively.

Let $\bb = \{b_1,\dots,b_m\}$ be a $w$-reduced set. Then, for any $a_1,\dots,a_m\in K^*$, the set $\{a_1b_1,\dots,a_mb_m\}$ is $w$-reduced.

Clearly, any (semi-) reduced set $\bb\subset L$ can be normalized to a (semi-) orthonormal set $\{\pi^{m_b}b\mid b\in \bb\}$, where $m_b:=-\lfloor w(b) \rfloor$.
\subsection{$\P$-reduceness}
In this section we introduce the notion of $\P$-reduceness, which can be seen as a ``local" concept of reduceness. Let $\P$ be a prime ideal of $\oo$ lying over $\p$ with ramification index $e:=e(\P/\p)$. We denote by $v_\P$ the discrete valuation which is induced by $\P$. Then, the discrete valuation 
$$
w_\P:L\rightarrow  \Q\cup \{\infty\},\quad  w_\P:=\frac{v_\P(z)}{e}
$$
becomes a $\p$-norm on $L$.

Clearly, $w_\P$ has a better behavior with respect to multiplications. In fact, we have 
$w_\P(ab)=w_\P(a)+w_\P(b)$, for all $a,b\in L$. Hence, we obtain the following statement.
\begin{lemma}\label{multstillred}
Let $\bb\subset L$ be a $w_\P$-reduced set and $c\in L^*$. Then, $c\bb$ is $w_\P$-reduced.
\end{lemma}

In order to derive an adequate reduceness criterion for the concept of ``local" reduceness, we are going to consider local fields. Therefore, we consider the completion of the field $L$ at the prime ideal $\P$ (details can be found in \cite{NK}). 

Denote by $L_\P$ the completion of $L$ at the prime ideal $\P$. Regarding the notation from Subsection \ref{divPol}, for $\P|\p$, we can realize the completion $L_\P$ as
\begin{align}\label{completiondef}
L_\P=K_\p(\t_\P),
\end{align}
where $K_\p$ is the completion of $K$ at $\p$, and $\t_\P$ denotes a root of $f_\P$, the irreducible factor of $f$ in $\hat{A}_\p[x]$ corresponding to $\P$.

The valuation $v_\p$ extends in a unique way to a non-discrete valuation 
$$
\hat{v}:\overline{K}_\p \rightarrow \Q\quad.
$$ 
Note that $\hat{v}(\iota_\P(z))=w_\P(z)$, for $z\in L$, where $\iota_\P$ denotes the injection of $L$ into $L_\P$ determined by $\t\mapsto\t_\P$.
In particular, it holds 
\begin{align}\label{valequ}
w_\P(g(\t))=\hat{v}(g(\t_\P))=\hat{v}(\res(g,f_\P))/n_\P,\quad  \text{ for all }g(x)\in A[x].
\end{align}
We denote by $\hat{\oo}_\P\subset L_\P$ the valuation ring of the restriction of $\hat{v}$ to $L_\P$ and set $\m_\P:=\{z\in \hat{\oo}_\P\mid \hat{v}(z)>0  \}$ the maximal ideal of $\hat{\oo}_\P$.

The next lemma will play a fundamental role in the subsequent description of the computation of $\p$-integral bases. Any nonzero prime ideal $\P$ of $\oo$ determines a set $\bb_\P$ of divisor polynomials (cf. Definition \ref{divpolset}). 
For a subset $\bb$ of $A[x]$ denote $\bb(\theta):=\{g(\t)\mid g\in \bb\}$.

\begin{lemma}\label{localbasered}
The set $\bb_\P(\t)$ is $w_\P$-reduced.
\end{lemma}
  
\begin{proof}
Clearly, $w:=w_\P=e^{-1}v_\P$ is a discrete valuation, and $w(g(\t))=\hat{v}(g(\t_\P))$ for all $g\in A[x]$ by (\ref{valequ}). Suppose that $\bb_\P(\t)=\{1,g_1(\t),\dots, g_{n_\P-1}(\t)\}$ is not $w$-reduced. Let $\lambda_0,\dots,\lambda_{n_\P-1}\in A$ with 
\begin{align}\label{inlem414}
w\Big(\sum_{i=0}^{n_\P-1}\lambda_ig_i(\t)\Big)>\min_{0\leq i< n_\P}\{w(\lambda_ig_i(\t))\}. 
\end{align}
By the strict triangle inequality we only have to consider all summands on the left hand side of (\ref{inlem414}), which have the same (minimal) $w$-value. In other words, we can assume that all summands on the left hand side of (\ref{inlem414}) have the same norm. According to Lemma \ref{divisorpolred} it holds $w(g_i(\t))\geq w(g_j(\t))$, for all $0\leq j<i<n_\P$. Since all summands in (\ref{inlem414}) have the same norm, we have $v_\p(\lambda_j)\geq v_\p(\lambda_i)$, for all $0\leq j<i<n_\P$. Hence, $g(x):=\lambda_{n_\P-1}^{-1}\sum_{i=0}^{n_\P-1}\lambda_ig_i(\t)$ is a monic polynomial of degree $n_{\P}-1$ with coefficients in $A_\p$ satisfying:
$$
w(g(\t))>  \min_{0\leq i< n_\P}\{w(\lambda_{n_\P-1}^{-1}\lambda_ig_i(\t))\}= w(g_{n_\P-1}(\t)),
$$
which is a contradiction, as $g_{n_\P-1}$ is a divisor polynomial of $f_\P$ (cf. Proposition \ref{propdivpol}).

\end{proof}

Henceforth we consider the $\p$-norm $w:=w_\P$. We are interested in a criterion to check wether a set $\bb\subset L$ is $w$-reduced or not and consider therefor a kind of reduction map. Let $\pi_\P$ be a prime element of the prime ideal $\P$ (i.e. $v_\P(\pi_\P)=1$). For any $r\in\Q$ the sets 
$$
L_{\ge r}:=\{z\in L\mid w(z)\ge r\}\supset L_{>r}:=\{z\in L\mid w(z)> r\}
$$
are $A_\p$-submodules of $L$. Their quotient is a $k_\p$-vector space $V_r:=L_{\ge r}/L_{>r}$, where $k_\p:=A/\p$ is the residue field of $\p$. For $r\notin w(L)$ it holds $V_r=0$, whereas for $r\in w(L)$ there is a non-canonical isomorphism $V_r\cong k_\P$, which we are going to describe. Suppose that $r\in w(L)$; that is, $er\in \Z$. Consider the division with remainder 
$$
er=qe+m,\quad 0\leq m<e.
$$
For $z\in L$ with $w(z)\geq r$, a reduction map is given by
\begin{align*}
\red^r_\P(z):L_{\ge r}\rightarrow k_\P,\quad z\mapsto z\pi^{-q} \pi_\P^{-m} +\P,
\end{align*}
The map $\red^r_\P$ induces a $k_\p$-linear isomorphism between $V_r$ and $k_\P$ and vanishes on $L_{>r}$.

\begin{theorem}\label{localPredcri}
A set $\bb\subset L$ is $w$-reduced if and only if for any $\rho\in \mathcal{R}:=\{w(b) + \Z\mid b\in\bb\}$ the vectors in
$$
\{\red^{w(b)}_{\P}(b)\mid b \in \bb \text{ with }w(b)+\Z=\rho\}
$$
 are $k_\p$-linearly independent.
\end{theorem}
\begin{proof}

For $\rho\in \mathcal{R}$, we set $\bb_\rho:=\{b\in \bb\mid w(b)+\Z = \rho \}$. We use the following claim in order to prove the statement.\medskip\\
\underline{\textbf{Claim}:}\smallskip\\
The set $\bb$ is $w$-reduced if and only if $\bb_\rho$ is $w$-reduced, for all $\rho\in \mathcal{R}$. \smallskip\\

By the claim we can assume that all vectors $b\in\bb$ have the same norm $\rho$ modulo $\Z$. Moreover, we may assume that all vectors $b\in\bb$ have the same norm, by replacing each $b\in \bb$ by $\pi^mb$ for an adequate choice of $m\in \Z$. Let us denote by $r:=w(b)$ this common norm.

Let $\bb=\{b_1,\dots,b_n\}$. We may consider $h_1,\dots,h_n\in A$ not all divisible by $\p$, so that $\min_{1\leq i\leq n}\{v_\p (h_ib_i)\}=r$. Then, trivially, $w(z)=r$ iff $\red^r_\P(z)\neq 0$ iff $\sum_{i=1}^n\overline{h}_i\red^r_\P(b_i)\neq 0$, where $\overline{h}\in k_\p$ indicates reduction modulo $\p$.

We prove the claim. Since any subset of a reduced family is reduced, we only need to show that $\bb$ is reduced if all $\bb_\rho$ are reduced.

Let $I:=\{\rho\in\Q/\Z\mid \bb_\rho\ne\emptyset\}$.
We have $E:=\langle \bb \rangle_K=\bigoplus_{\rho \in I}E_\rho$, where $E_\rho$ is the subspace of $E$ generated by $\bb_\rho$.
Take $a_1,\dots,a_n\in K$ and let $x=\sum_{i=1}^n a_ib_i$. This element splits as $x=\sum_{\rho\in I} x_\rho$, where $x_\rho=\sum_{b_i\in \bb_\rho}a_ib_i$. 
Since all values $w(x_\rho)$ are different (because $w(a_ib_i)\equiv w(b_i)\bmod{\Z}$),
we have $w(x)=\min_{\rho\in I}\{w(x_\rho)\}$. On the other hand, since all $\bb_\rho$ are reduced, we have $w(x_\rho)=\min_{b_i\in \bb_\rho}\{w(a_ib_i)\}$. Thus, $\bb$ is reduced.

\end{proof}
In order to obtain an analogous criterion to test if a subset of $L$ is $w$-semi-reduced we introduce another kind of reduction map. To this end, we consider $\P$-adic expansions of elements in $ L_\P$. Let $\oo_\P$ bee valuation ring of $v_\P$ and fix a system of representatives $R$ of the residue class field $k_\P:=\oo_\P/\P\cong\hat{\oo}_\P/\m_\P$ of $\P$. By \cite[Satz 4.4]{NK} any nonzero element $z$ in $L_\P$ has a unique representation $z=\pi_\P^m(\lambda_0+\lambda_1\pi_\P+\lambda_2\pi_\P^2+\cdots)$, where $\lambda_i\in R$ and $m=v_\P(z)\in\Z$. In particular, for any $z\in L^*$ we can write
\begin{align*}
\iota_\P(z)=\sum_{j=v_\P(z)}^\infty \lambda_j \pi_\P^j,\quad  \lambda_j\in R.
\end{align*}

\begin{definition}
For $r\in \Q$ and $z\in L_{\geq \lfloor r\rfloor}$, we have $v_\P(z\pi^{-\lfloor r\rfloor})=v_\P(z)-e\lfloor r\rfloor\geq 0$. We write $\iota_\P(z\pi^{-\lfloor r\rfloor})=\lambda_0+\lambda_1\pi_\P+\cdots+\lambda_{e-1}\pi_\P^{e-1}+\cdots$, and define 
$$
\sred^r_\P(z):=
(\overline{\lambda}_{0},\dots, \overline{\lambda}_{e-1})\in k_\P^{e},
$$ 
where $\overline{\lambda}\in k_\P$ is reduction modulo $\P$ of $\lambda$.
\end{definition}
Clearly, $\sred^r_\P$ induces a $k$-linear mapping $V_{\lfloor r\rfloor}\rightarrow k_\P^e$.

\begin{theorem}\label{semiredcriterio}
A set $\bb\subset L$ is $w$-semi-reduced if and only if the vectors in
$$
\{\sred^{w(b)}_\P(b)\mid b \in \bb \}
$$
are $k_\p$-linearly independent.
\end{theorem}
\begin{proof}
The statement can be proven by considering the proof of Theorem \ref{localPredcri} and replacing $w$ by $\lfloor w\rfloor$ and $\red^{r}_\P$ by $\sred^{r}_\P$, respectively.
\end{proof}

\subsection{$\p$-reduceness}

The concept of $\P$-reduceness can be generalized to several prime ideals $\P_1,\dots,\P_s$. Henceforth denote by $S=\{\P_1,\dots,\P_s\}$ the set of all prime ideals lying over $\p$. The set $S$ induces a mapping
\begin{align}\label{ws}
w_S:L\rightarrow \Q\cup\{\infty\}, \quad w_S(z):=\min_{1\leq i\leq s}\{w_{\P_i}(z)\}
\end{align}

An immediate consequence of this definition is the following observation.

\begin{lemma}
The map $w_S$ is a $\p$-norm on $L$.
\end{lemma}

As in the last subsection we define ``reduction maps" $\red$ and $\sred$ in order to generalize the reduceness-criterion from Theorem \ref{localPredcri} and the semi-reduceness-criterion of Theorem \ref{semiredcriterio} to this situation. Denote for $1\leq i\leq s$ by $e_i:=e(\P_i/\p)$ the ramification index of $\P_i$ and set $w:=w_S$.

\begin{definition}
For $r\in \Q$ and $z\in L$, we define 
\begin{align*}
\red^r_S(z)&:=(\red^r_{\P_i}(z))_{1\leq i\leq s}\in k_{\P_1}\times \cdots \times k_{\P_s}\text{ and} \\
\sred^r_{S}(z)&:=(\sred^r_{\P_i}(z))_{1\leq i\leq s}\in k_{\P_1}^{e_1}\times \cdots \times k_{\P_s}^{e_s}.
\end{align*}

\end{definition}
 
The following properties are transmitted by the properties of the local mappings $\red^r_{\P_i}$ and $\sred^r_{\P_i}$, for $1\leq i\leq s$, respectively.

\begin{lemma}\label{transmitprop}
The mappings $\red^r_S$ and $\sred^r_S$ induce $k_\p$-linear mappings $V_r\rightarrow k_{\P_1}\times\cdots \times k_{\P_s}$ and $V_{\lfloor r\rfloor}\rightarrow k_{\P_1}^{e_1}\times\cdots\times k_{\P_s}^{e_s}$ and vanish on $L_{>r}$ and $L_{\geq\lfloor r\rfloor+1}$, respectively.
\end{lemma}

Analogously to Theorem \ref{localPredcri} and Theorem \ref{semiredcriterio} one can prove the following statements.
 
\begin{theorem}\label{redcrigen}
A set $\bb\subset L$ is $w$-reduced if and only if for any $\rho\in\mathcal{R}:=\{w(b) + \Z\mid b\in\bb\}$ the vectors in
$$
\{\red^{w(b)}_S(b)\mid b \in \bb \text{ with }w(b)+\Z=\rho\}
$$
are $k_\p$-linearly independent.
\end{theorem}

\begin{theorem}\label{semiredcriglobal}
A set $\bb\subset L$ is $w$-semi-reduced if and only if the vectors in
$$
\{\sred^{w(b)}_S(b)\mid b \in \bb \}
$$
are $k_\p$-linearly independent.
\end{theorem}

 We are interested in the relation between $\P$-reduceness and $S$-reduceness. Recall that $\pi$ is a fixed uniformizer of $\p$.

\begin{theorem}\label{intbase1}
For $1\leq i\leq s$, let $\bb_{i}\subset L$ be $w_{\P_i}$-reduced and $z_i\in L$ such that, for all $b\in\bb_i$,
\begin{align}\label{ineqcon}
 w_{\P_i}(z_i b)<  w_{\P_j}(z_i b),  \text{ for } j \in \{1,\dots,s\}\setminus \{i\}.
\end{align}
Then, $\bigcup_{i=1}^sz_i\bb_i$ is $w$-reduced.
\end{theorem}
\begin{proof}
We set $\red^r_i:=\red^r_{\P_i}$ and $w_i:=w_{\P_i}$, for $1\leq i\leq s$ and $r\in \Q$. By (\ref{ineqcon}), for $1\leq i\leq s$ and $b\in \bb_i$, it holds $w(z_ib)=w_i(z_ib)<w_j(z_ib)$, for $j \neq i$. Then, we obtain $\red_j^{w(z_ib)}(z_ib)=0\in k_{\P_j}$, for all $j\neq i$, by the definition of $\red_j^r$. Hence, $\red_S^{w(z_ib)}(z_ib)$ is given by
$$
(0,\dots,0,\red_i^{w(z_ib)}(z_ib),0,\dots,0).
$$
By Lemma \ref{multstillred} the sets $z_i\bb_i$ are $w_i$-reduced, for $1\leq i\leq s$, and therefore, for each $\rho\in \Q/\Z$, the elements $\{\red_S^{w(z_ib)}(z_ib)\mid b\in\bb_i,\quad w(b)+\Z=\rho\}$ are $k_\p$-linearly independent by Theorem \ref{localPredcri}. Hence, the set $\bigcup_{i=1}^sz_i\bb_i$ is $w$-reduced by Theorem \ref{redcrigen}.
\end{proof}

\begin{theorem}\label{globalargument}
For $1\leq i\leq s$, let $\bb_{i}\subset L$ be $w_{\P_i}$-reduced and $z_i\in L$ such that for all $b\in\bb_i$
\begin{enumerate}
\item $\lfloor w_{\P_i}(z_i b)\rfloor\leq \lfloor w_{\P_j}(z_i b)\rfloor$ for $1\leq i<j\leq s$
\item $\lfloor w_{\P_i}(z_i b)\rfloor<\lfloor w_{\P_l}(z_i b)\rfloor$ for $1\leq l<i\leq s$.
\end{enumerate}
Then, $\bigcup_{i=1}^sz_i\bb_i$ is $w$-semi-reduced.
\end{theorem}
\begin{proof}
We set $w_i:=w_{\P_i}$ and $\sred_i^r:=\sred_{\P_i}^r$, for $1\leq i\leq s$ and $r\in \Q$. By the hypothesis, for $1\leq l<i\leq s$ and $b\in \bb_i$, we have $\lfloor w(z_ib) \rfloor=\lfloor w_i(z_ib) \rfloor< \lfloor w_l(z_ib) \rfloor$; hence, $\sred_l^{w(z_ib)}(z_ib)=0\in k_{\P_l}^{e_l}$.
In particular, with $r_i:=w(z_ib)$ we deduce, for $1\leq i\leq s$,
\begin{align}
\sred_{S}^{r_i}(z_ib)=(0,\dots,0,\sred_{i}^{r_i}(z_ib),*,\dots,*),\label{2.eqinprof}
\end{align}
with some vectors $*\in k_{\P_j}^{e_j}$, for $j>i$. Since $\lfloor w(z_ib) \rfloor=\lfloor w_i(z_ib) \rfloor$, we have $\sred^{w(z_ib)}_i(z_ib)=\sred_i^{w_i(z_ib)}(z_ib)$, for $b\in \bb_i$. According to Lemma \ref{multstillred} the sets $z_i\bb_i$ are $w_{i}$-reduced, and particularly $w_{i}$-semi-reduced. Then, by Theorem \ref{semiredcriterio} the family $\{\sred_i^{w_i(z_ib)}(z_ib)\mid b\in\bb_i\}$ is $k_\p$-linearly independent. By (\ref{2.eqinprof}), the family $\bigcup_{1\leq i\leq s}\{\sred^{w(z_ib)}_S(z_ib)\mid b\in\bb_i\}$ is $k_\p$-linearly independent. Thus, by Theorem \ref{semiredcriglobal} $\bigcup_{1\leq i\leq s} z_i\bb_i$ is $w$-semi-reduced.

\end{proof}

\section{Computation of $\p$-integral bases}\label{p_intBase}

Let $\t$ be a root of a monic irreducible separable polynomial $f\in A[x]$ of degree $n$ and let $L=K(\t)$ be the finite separable extension of $K$ generated by $\t$. We fix a non-zero prime ideal $\p$ of $ A$ and denote by $A_\p$ the localization of $A$ at $\p$ and set $k_\p=A/\p$. We denote by $\oo$ the integral closure of $A$ in $L$. The goal of this section is to describe an algorithm, which computes a (reduced) $\p$-integral basis of $\oo$.

\begin{lemma-definition}\label{DefpInt}
Let $b_1,\dots,b_n \in \oo$ be $A$-linearly independent elements and denote by $M= \gen{b_1,\dots,b_n}_A$ the $A$-submodule of $\oo$ that they generate. The following conditions are equivalent:
\begin{enumerate}
\item $ b_1,\dots, b_n$ are an $A_\p$-basis of $\oo_\p$.
\item $ b_1,\dots, b_n$ are a $k_\p$-basis of $\oo_\p/\p\oo_\p$.
\end{enumerate}
If these conditions are satisfied we call $ (b_1,\dots, b_n)$ a $\p$-integral basis of $\oo$.\medskip
\end{lemma-definition}
\begin{proof}
The two conditions are equivalent by Nakayama's lemma.
\end{proof}

Denote by $S$ the set of all prime ideals of $\oo$ lying over $\p$ and consider the $\p$-norm $w_S$ defined in (\ref{ws}). In \cite{HN} it was shown that a $w_S$-orthonormal set of $n$ elements in $L$ determines a $\p$-integral basis. The next theorem is an improvement of this result and a new characterization of $\p$-integral bases.

\begin{theorem}\label{newdesofint}
Let $\bb$ be subset of $L$ with $n$ elements. Then, $\bb$ is a $\p$-integral basis of $\oo$ if and only if $\bb$ is $w_S$-semi-orthonormal.
\end{theorem}

In order to prove Theorem \ref{newdesofint} we will use the following lemma.

\begin{lemma}\label{helpLem}
Let $\bb'=(b_1',\dots,b'_n)$ be a $w$-semi-orthonormal basis, $T\in\gl_n(A_\p)$, and $\bb=(b_1,\dots,b_n)$ determined by $(b_1',\dots,b'_n)T=(b_1,\dots,b_n)$. Then, $\bb$ is a $w$-semi-orthonormal basis.\\
\end{lemma}
\begin{proof}
We consider the extension of $v_\p$ to $K^n$:
$$v_\p((a_1,\dots,a_n))=\min_{1\leq i\leq n}\{v_\p(a_i)\}.$$ 
\underline{\textbf{Claim}:}\smallskip\\
A matrix $T=(t_{i,j})\in K^{n\times n}$ belongs to $\gl_n(A_\p)$ if and only if preserves $v_\p$; that is, $v_\p(aT)=v_\p(T)$ for all $a\in K^n$.\\

After the claim, the statement of the lemma yields immediately: For $a=(a_1,\dots,a_n)$, $a'=(a'_1,\dots,a'_n)$ with $a'_i:=\sum_{j=1}^nt_{i,j}a_j$, and for a $w$-semi-orthonormal basis $\bb'$ it holds
\begin{align*}
\Big\lfloor w\Big(\sum_{i=1}^na_ib_i\Big)\Big\rfloor=\Big\lfloor w\Big(\sum_{i=1}^na'_ib'_i\Big)\Big\rfloor&=\min_{1\leq i\leq n}\{v_\p(a'_i)\}=v_\p(aT)=v_\p(a),
\end{align*}
Thus, $\bb$ is $w$-semi-orthonormal.

In order to prove the claim suppose $T$ preserves $v_\p$. For all vectors $e_i$ of the standard basis of $K^n$ we have $v_\p(e_iT)=v_\p(T)=0$, so that all rows of $T$ have entries in $A_\p$. This shows that $T\in A_\p^{n\times n}$. The reduction $\overline{T}\in k_\p^{n\times n}$ acts on $k_\p^n$ and sends non-zero vectors to non-zero vectors. Thus, $\overline{T}$ is invertible and therefore $T\in\gl_n(A_\p)$.

Now, assume $T\in\gl_n(A_\p)$. Then $T$ is a product of elementary matrices. Since elementary matrices preserve $v_\p$, $T$ has the same property.

\end{proof}

\begin{proof}[Proof of Theorem \ref{newdesofint}]
We set $w:=w_S$. Suppose that $\bb$ is $w$-semi-orthonormal. An easy computation shows that $w(z)\geq 0$ if and only if $z\in \oo_\p$; hence, the set $\bb$ is a subset of $\oo_\p$. According to Lemma-Definition \ref{DefpInt} it is sufficient to show that $\bb$ is a set of $k_\p$-linearly independent vectors in order to show that $\bb$ is a $\p$-integral basis of $\oo$.

Assume $\sum_{b\in\bb}\lambda_bb\in \p\oo\otimes_A A_\p$ with $\lambda_b\in A_\p$. Then, $w\big(\sum_{b\in\bb}\lambda_bb\big)\geq 1$. Since $\bb$ is $w$-semi-orthonormal, we deduce $w(\lambda_bb)\geq 1$, for all $b\in\bb$, and therefore $v_{\p}(\lambda_b)\geq 1$, for all $b\in \bb$. That is, $\lambda_b\in \p A_\p$, for $b\in \bb$.

For the other direction let $\bb$ be any $\p$-integral basis of $\oo$. In \cite{HN} it is shown that a (semi-) reduced basis $\bb'$ of $L$ exists. We can assume that $\bb'$ is already normalized to a $w$-semi-orthonormal subset of $\oo$ with $n$ elements. As shown above, the family $\bb'$ is also a $\p$-integral basis of $\oo$; hence, the transition matrix from $\bb$ to $\bb'$ belongs to $\gl_n(A_\p)$. Thus, Lemma \ref{helpLem} states that $\bb$ is $w$-semi-orthonormal too.

\end{proof}
\subsection{The algorithm}

Let $S=\{\P_1,\dots,\P_s\}$ be the set of all prime ideals of $\oo$ lying over $\p$. For $1\leq j\leq s$, we denote by $\Phi_j:=\phi_{\P_j}$ an Okutsu approximation of the $\p$-adic irreducible factor $f_{\P_j}$ of $f$ in $\hat{A}_\p[x]$ (cf. subsection \ref{OkApp}) and let $\bb_j:=\bb_{\P_j}(\t)$, where $\bb_{\P_j}$ is defined in Definition \ref{divpolset}. 

In \cite{NewC} are given closed formulas for the values $v_{\P_j}(\Phi_i(\t))$ for $j\neq i $ in terms of data collected by the Montes algorithm. We recall that  when we improve $\Phi_i$, the value $v_{\P_i}(\Phi_i(\t))$ increases, but the values $v_{\P_j}(\Phi_i(\t))$ for $j\neq i$ remain constant.

We set $n_{\P_j}=\deg f_{\P_j}$, for $1\leq j\leq s$. Note that $n_{\P_j}=e(\P_j/\p)f(\P_j/\p)$, where $f(\P_j/\p)$ denotes the residue degree of $\P_j$ over $\p$. By Lemma \ref{localbasered} the set $\bb_j$ is $w_{\P_j}$-reduced.
By definition it holds $\#\bb_j=\deg f_{\P_j}$, for $1\leq j\leq s$; hence, $\#\bigcup_{\kappa=1}^s\bb_\kappa=\sum_{i=1}^s\deg f_{\P_j}=\deg f=n$. Denote by $\pi$ a uniformizer of $\p$.

By applying Theorems \ref{intbase1}, \ref{globalargument}, and \ref{newdesofint} we obtain the next two statements.

\begin{theorem}\label{INtBasis}
For $1\leq \kappa\leq s$, we set 
\begin{align}\label{multichoice}
z_\kappa:=\prod\limits_{\substack{ j=1\\ j\neq \kappa}}^s \Phi_j^{\epsilon_j}(\t),
\end{align}
where $\epsilon_j\in\{0,1\}$ and the Okutsu approximation $\Phi_j$ are chosen in such a way that, for all $b\in\bb_\kappa$,
\begin{enumerate}
\item $ \lfloor\ww{z_\kappa b}_{{\P_\kappa}}\rfloor\leq  \lfloor\ww{z_\kappa b}_{{\P_i}}\rfloor$, for $1\leq \kappa<i\leq s$,
\item $ \lfloor\ww{z_\kappa b}_{{\P_\kappa}}\rfloor<  \lfloor\ww{z_\kappa b}_{{\P_i}}\rfloor$, for $1\leq i<\kappa\leq s$.
\end{enumerate}
Then, $\{b_1,\dots,b_n\}:=\bigcup_{\kappa=1}^sz_\kappa\bb_\kappa$ is $w_{S}$-semi-reduced. In particular, the family
$$
\frac{b_i}{\pi^{\lfloor w_S(b_i)\rfloor}},\quad 1\leq i\leq n,
$$
is a $\p$-integral basis of $\oo$.
\end{theorem}

\begin{theorem}\label{INtBasisred}
If we replace in Theorem \ref{INtBasis} item \textit{1} and \textit{2} by the condition 
$$ 
w_{{\P_\kappa}}(z_\kappa b) <  w_{{\P_i}}(z_\kappa b),  \text{ for } i \in \{1,\dots,s\}\setminus \{\kappa\},
$$
then $(b_i/\pi^{\lfloor w_S(b_i)\rfloor})_{1\leq i\leq n}$ is a $w_S$-orthonormal basis of $\oo$.
\end{theorem}
The idea of using multipliers to construct integral bases goes back to Ore (1925). In \cite{HN} a similar way of determining adequate multipliers is presented. An advantage of our choice is that in practice the multipliers $z_\kappa$ are simple. That is, many exponents $\epsilon_j$ in (\ref{multichoice}) can be chosen to be zero. Often we may take
$$
z_\kappa=\prod\limits_{j<\kappa} \Phi_j(\t),\quad 1\leq\kappa\leq s.
$$
Since $\deg \Phi_j=n_{\P_j}$ and $\bb_{\P_j}=\{g_0(\t),\dots\,g_{n_{\P_j}}(\t)\}$ with $g_m\in A[x]$ monic of degree $m$ for $1\leq j\leq s$, the degree of $\prod_{j<\kappa} \Phi_j(x)g_m$ is equal $\sum_{j<\kappa}n_{\P_j}+m$, and the basis $\bb$ is in that particular case triangular. Even though, the multipliers $z_\kappa$ are not always that simple, our choice leads in many cases to a partly triangular basis $\bb$. Hence, the resulting $\p$-integral basis $\big(b_i/\pi^{\lfloor w_S(b_i)\rfloor}\big)_{ 1\leq i\leq n}$ can be transformed quickly into a triangular or Hermite basis.

\begin{algorithm}
\caption{: Computation of a $\p$-integral basis}\label{AlgopINT}
\begin{algorithmic}[1]

\REQUIRE Monic separable and irreducible polynomial $f\in A[x]$, an uniformizer $\pi$ of a non-zero prime ideal $\p$ of $A$, and a boolean variable $\mathrm{red}$.
\ENSURE A $\p$-integral basis of $\oo$, which is additionally $w_{S}$-orthonormal if $\mathrm{red}=\mathrm{TRUE}$. 

\STATE Algorithm \ref{AlgoMontes}($f$,\,$\pi$)
\FOR{$\P_i|\p$}

\STATE Determine $\bb_i=\bb_{\P_i}(\t)$ with $\bb_{\P_i}$ as in Definition \ref{divpolset}
\STATE Determine $\Phi_i$ by Algorithm \ref{SFL} satisfying, if $\mathrm{red}=\mathrm{FALSE}$, the conditions of Theorem \ref{INtBasis} and else the conditions of Theorem \ref{INtBasisred}
\ENDFOR
\STATE $\{b_1,\dots,b_n\}\leftarrow  \bigcup_{\kappa=1}^sz_\kappa\bb_\kappa$
\STATE \textbf{return} $\left( b_i/\pi^{\lfloor \ww{b_i}_{S}\rfloor}\right)_{1\leq i\leq n}$

\end{algorithmic}
\end{algorithm}

In order to determine the exponents $\epsilon_j$ and the precision of the approximations $\Phi_j$ so that $z_\kappa$ satisfies the conditions of Theorem \ref{INtBasis} or Theorem \ref{INtBasisred}, we have to compute the values $\ww{z_\kappa b}_{{\P_j}}$, for all $1\leq \kappa,j\leq s$ and $b\in \bb_{\kappa}$. That is, we need to determine the values $v_{\P_\kappa}(\Phi_{j}(\t))$ and $v_{\P_j}(b)$, for $1\leq \kappa,j\leq s$ and all $b\in \bb_\kappa$. In \cite[Proposition 4.7]{NewC} concrete formulas can be found, which only depend on the data computed along Algorithm \ref{AlgoMontes}. Hence, these values can be computed as a by-product at cost zero. Thus, the cost of the determination of the integers $\epsilon_j$ and the precision of the approximations $\Phi_j$ can be neglected. 


In order to determine an integral basis of $\oo$ (i.e. an $A$-basis), we may compute $\p$-integral bases $\bb_{\p}$ of $\oo$ for any prime ideal $\p$ with $v_\p(\dsc f)> 1$ and transform it into a triangular basis. Then by an easy application of the CRT one can combine the ``local" bases $\bb_\p$ to a global one.

\subsection{Complexity}

For the subsequent complexity analysis we define $\delta:=v_\p (\dsc f)$ the $\p$-valuation of the discriminant of $f$. The following steps dominate the runtime of Algorithm \ref{AlgopINT}:
\begin{enumerate}
\item Montes algorithm \label{en1}
\item Computation of local sets $\bb_i=\bb_{\P_i}(\t)$, for $1\leq i\leq s$ \label{en2}
\item  Computation of multiplier $z_i$, for $1\leq i\leq s$:\label{en3}
\begin{enumerate}
\item Determining $\Phi_i$ with necessary precision \label{en31}
\item Computing $z_\P:=\prod_{1\leq i\leq s} \Phi_i^{\epsilon_i}$\label{en32}
\end{enumerate}
\item Computing $  \bigcup_{1\leq i\leq s}z_i\bb_i$\label{en4}
\end{enumerate}

We admit fast multiplication techniques of Sch\"onhage-Strassen \cite{Gath2}. Let $R$ be a ring and let $g_1,g_2\in R[x]$ be two polynomials, whose degrees are bounded by $d_1$ and $d_2$, respectively. Then, the multiplication $g_1\cdot g_2$ needs at most $O(\max\{d_1,d_2\}^{1+\epsilon})$ operations in $R$. We may consider the elements in $A$ to be finite $\pi$-adic developments whose length is at most $\delta+1$ by \cite[Thm. 3.14]{BNS}. We fix a system of representatives $\mathcal{R}$ of $A/\p$ and call an operation in $A$ \emph{$\p$-small} if it involves two elements belonging to $\mathcal{R}$. Hence, any multiplication in $A$ can be realized with at most $O(\delta^{1+\epsilon})$ $\p$-small operations.
We assume that the residue field $A/\p$ is finite with $q$ elements. \\

\noindent\textbf{\ref{en1}. Montes algorithm}:

The Montes algorithm has a cost of $O\left(n^{2+\epsilon}+n^{1+\epsilon}\delta\log q+n^{1+\epsilon}\delta^{2+\epsilon}\right)$
$\p$-small operations \cite[Thm. 5.15]{BNS}. \\

\noindent\textbf{\ref{en2}. Local sets}:

We begin with analyzing the cost of determing $\bb_i=\bb_{\P_i}(\t)$ for one $i\in \{1,\dots,s\}$ as defined in Definition \ref{divpolset}. We fix $\P=\P_i$ corresponding to the type $\ti=\ti_\P$ and consider $g_m(x):=\prod_{i=0}^r\phi_i(x)^{c_i}$ with $m=\sum_{i=0}^rc_im_i,\quad 0\leq c_i<m_{i+1}/m_i$ for $0<m<n_\P$. Let $d_i:=m_{i+1}/m_i-1$ and define $g(x):=\prod_{i=0}^r\phi_i(x)^{d_i}$. Then, $g_m(x)| g(x)$ for $0<m<n_\P$. Thus, the cost of computing $g$ by brute force is dominating the complexity of the computation of $g_0,\dots, g_{n_\P-1}$. For any power $\phi_i^{d_i}$ in $g$ we count $d_i-1$ multiplications in $A[x]$. Since $g$ is the product of $r+1$ powers we can determine $g$ by 
$$\sum_{i=0}^rd_i-1+r=\left(\sum_{i=0}^rd_i\right)-1<\sum_{i=0}^r(m_{i+1}/m_i)=\sum_{i=0}^re_if_i=n_\P$$ 
multiplications in $A[x]$. Hence we can compute $\bb_\P$ with at most $n_{\P}$ multiplications in $A[\t]$. Thus, the cost of the computation of $\bb_{\P_j}$ is $n$ multiplications in $A[\t]$ or equivalently, $O\left(n^{2+\epsilon}\delta^{1+\epsilon}\right)$ $\p$-small operations.\\

\noindent\textbf{\ref{en31} Okutsu approximation}:

In order to compute the multipliers $z_\kappa$ defined in (\ref{multichoice}), we have to improve the Okutsu approximations $\Phi_i$ to an adequate precision. According to \cite[Theorem 5.16]{BNS}, the cost of the computation of an Okutsu approximation $\Phi_i$ with precision $\nu$ at $\P_i$ (that is, $w_{\P_i}(\Phi_i(\t))\geq \nu$) is given by
$$
O(nn_{\P_i}\nu^{1+\epsilon}+n\delta^{1+\epsilon})
$$ 
$\p$-small operations.

The following technical lemmas provide concrete bounds for the precision $\nu$ of the Okutsu approximation $\Phi_i$, for $1\leq i\leq s$, which is sufficient in order to determine a $\p$-integral basis with Algorithm \ref{AlgopINT}. 

In the following observation we assume that the multipliers $z_\kappa$ are given by 
$$
z_\kappa=\prod\limits_{\substack{ j=1\\ j\neq \kappa}}^s \Phi_j^{\epsilon_j}(\t),\quad \text{all }\epsilon_j=1.
$$
Although in practice many of the exponents $\epsilon_j$ can be chosen to be zero, for the complexity estimation we consider the worst case $\epsilon_j=1$, for $j\neq \kappa$.
\begin{lemma}\label{BoundPrec}
For $1\leq i\leq s$, let $\bb_i=\{b_{i,j}\mid 0\leq j < n_{\P_i}\}$ and $\Phi_i$ such that
\begin{align}\label{crazycondi}
w_{\P_i}(\Phi_i(\t))\geq \max\Big\{\max_{1\leq \kappa<i\leq s}\{H_{i,\kappa}\},\max_{1\leq i<\kappa\leq s}\{H_{i,\kappa}\}+1\Big\}, 
\end{align}
where
$$
H_{i,\kappa}:=\max_{0\leq l< n_{\P_\kappa}}\Big\{w_{\P_\kappa}\Big( b_{\kappa,l}\prod_{ \substack{ j=1\\ j\neq \kappa}}^s\Phi_j(\t)\Big)- w_{\P_i}\Big({b_{\kappa,l}\prod_{ \substack{ j=1\\ j\neq \kappa,i}}^s\Phi_j(\t)}\Big) \Big\}.
$$
Then, $\{b_1,\dots,b_n\}=\bigcup_{\kappa=1}^sz_i\bb_\kappa$, with $z_\kappa:=\Phi_1(\t)\cdots \Phi_{\kappa-1}(\t)\cdot\Phi_{\kappa+1}(\t)\cdots \Phi_{s}(\t)$ is $w_S$-semi-reduced. In particular, the family
$$
\frac{b_i}{\pi^{\lfloor w_S(b_i)\rfloor}},\quad 1\leq i\leq n
$$
is a $\p$-integral basis of $I$. 
 \end{lemma}
\begin{proof}
We show that the conditions on the $\Phi_i$ can be translated to the following statement: For $1\leq \kappa\leq s$ and for $0\leq l<n_{\P_\kappa}$ it holds
\begin{enumerate}
\item $ w_{\P_\kappa}(z_\kappa b_{\kappa,l})\leq  w_{\P_i}(z_\kappa b_{\kappa,l})$ for $1\leq \kappa<i\leq s$ and
\item $  w_{\P_\kappa}(z_\kappa b_{\kappa,l})\leq  w_{\P_i}(z_\kappa b_{\kappa,l})+1$ for $1\leq i<\kappa\leq s$.
\end{enumerate}
Then, the statement of the lemma follows from Theorem \ref{INtBasis}.

The inequality $w_{\P_i}(\Phi_i(\t))\geq H_{i,\kappa}$, for $\kappa<i$, implies that, for $0\leq l<n_{\P_\kappa}$,
\begin{align*}
w_{\P_i}(\Phi_i(\t))&\geq w_{\P_\kappa}\Big( b_{\kappa,l}\prod_{ \substack{ j=1\\ j\neq \kappa}}^s\Phi_j(\t)\Big)- w_{\P_i}\Big(b_{\kappa,l}\prod_{ \substack{ j=1\\ j\neq \kappa,i}}^s\Phi_j(\t)\Big)\\
\Longleftrightarrow \quad w_{\P_i}\Big(b_{\kappa,l}\prod_{ \substack{ j=1\\ j\neq \kappa}}^s\Phi_j(\t)\Big)& \geq w_{\P_\kappa}\Big( b_{\kappa,l}\prod_{ \substack{ j=1\\ j\neq \kappa}}^s\Phi_j(\t)\Big),
\end{align*}
which proves the first item. Analogously, the inequality $w_{\P_i}(\Phi_i(\t))\geq H_{i,\kappa}+1$, for $\kappa>i$, implies the second item.
\end{proof}
Analogously to the last proof one can show with Theorem \ref{INtBasisred} the following statement.
\begin{corollary}If we require 
$$
w_{\P_i}(\Phi_i(\t))> \max_{1\leq\kappa\leq s}\{H_{i,\kappa}\mid \kappa\neq i\},
$$
for $1\leq i\leq s$, instead of (\ref{crazycondi}), then the $\p$-integral basis from the last lemma is $w_{S}$-orthonormal.
\end{corollary}

By Lemma \ref{BoundPrec} we deduce a lower bound for the precision of the approximations $\Phi_i$, for $1\leq i\leq s$. 
\begin{lemma}\label{hardcorelem}
For $1\leq i\neq\kappa\leq s$ we have
$$
H_{i,\kappa}=O(\delta).
$$
\end{lemma}
\begin{proof}
We keep the notation from Lemma \ref{BoundPrec}. For $1\leq \kappa,j\leq s$ and $0\leq l< n_{\P_\kappa}$ it holds 
$$
\ w_{\P_\kappa}( b_{\kappa,l}z_\kappa)- w_{\P_i}\Big(b_{\kappa,l}\prod_{ \substack{ j=1\\ j\neq \kappa,i}}^s\Phi_j(\t)\Big) 
\leq \ w_{\P_\kappa}( b_{\kappa,l}z_\kappa),
$$
since $b_{\kappa,l},\ \Phi_i(\t)\in \oo$. We estimate $w_{\P_\kappa}( b_{\kappa,l}z_\kappa)$ in order to determine a bound for $H_{i,\kappa}$. By definition, the elements $b_{\kappa,l}\in \bb_\kappa$ are given by $b_{\kappa,l}=g_{\kappa,l}(\t)$ with $g_{\kappa,l}(x)\in A[x]$ monic of degree $m<n_{\P_\kappa}$. In \cite[Proposition 1.3]{BNS} it is shown that all monic polynomials $g \in A[x]$ of degree less than $n_{\P_\kappa}$ satisfy $v_{\P_\kappa}(g(\theta))/e(\P_\kappa/\p)\leq \mu$ for a certain constant $\mu$ which satisfies $\mu\leq \delta/n_{\P_\kappa}$. Hence, $w_{\P_\kappa}(b_{\kappa,l})\leq  \delta/n_{\P_\kappa}$, for all $0\leq l< n_{\P_\kappa}$. 

We consider $w_{\P_\kappa}( z_\kappa)=\sum_{j=1,j\neq \kappa}^sv_{\P_\kappa}(\Phi_j(\t))/e(\P_\kappa/\p)$. Let $f_{1},\dots,f_{s}$ be the irreducible factors of the polynomial $f$ in $\kp[x]$. As in (\ref{completiondef}), we identify the completion of $L$ at the prime ideal $\P_\kappa$ with $\Kp(\t_{\P_\kappa})$, for $1\leq \kappa\leq s$, where $\t_{\P_\kappa}$ denotes a root of the irreducible factor $f_{\kappa}$. Let $\hat{v}$ be the extension of $v_{\p}$ to the algebraic closure of $\Kp$. With (\ref{valequ}) it holds
$$
\delta=\sum_{i=1}^sv_\p(\dsc (f_{\P_i}))+2\sum_{1\leq i<j\leq s}v_\p(\res(f_{\P_i},f_{\P_j}))
$$  
\cite[\textrm{III}.\textsection 2-4]{JPS}, and we deduce $w_{\P_\kappa}( z_\kappa)\leq \delta$. Together with the previous estimations, we obtain $H_{i,\kappa}=O(\delta)$.

\end{proof}
According to the last lemma, we compute in Algorithm \ref{AlgopINT} approximations $\Phi_i$ with a precision $\nu=O(\delta)$ at cost of
$$
O(nn_{\P_i}\delta^{1+\epsilon})
$$ 
$\p$-small operations. In the worst case we have to determine all approximations $\Phi_i$ with that precision. As $\sum_{i=1}^{s}n_{\P_i}=n$, the cost of computing the adequate approximations can be estimated by $O(n^2\delta^{1+\epsilon})$ $\p$-small operations. \\

\noindent\textbf{\ref{en32} Multiplier}:

We analyze the cost of the computation of the multipliers $z_\kappa$, for $1\leq \kappa\leq s$. As mentioned before, the worst case occurs if any multiplier $z_\kappa$ is given by
$$
z_\kappa=\prod_{ \substack{ j=1\\ j\neq \kappa}}^s\Phi_j(\t).
$$
\begin{lemma}
Let $s\geq 2$. The multipliers $z_1,\dots,z_s$ can be determined by $2(s-3)+s$ multiplications in $A[\t]$.
\end{lemma}
\begin{proof}
Initially we compute the products
\begin{align}
&\Phi_1\Phi_2,\Phi_1\Phi_2\Phi_3,\dots, \Phi_1\cdots\Phi_{s-2}\text{ and}\label{Liste1}\\
&\Phi_{s-1}\Phi_s,\Phi_{s-2}\Phi_{s-1}\Phi_s,\dots, \Phi_3\cdots\Phi_{s}.\label{Liste2}
\end{align}
This can be realized by $2(s-3)$ multiplications. Every $z_i$ can be written as a product of two factors, where each of them belongs to list (\ref{Liste1}), list (\ref{Liste2}), or is one of the $\Phi_i$. Hence, to determine the multipliers $z_1,\dots,z_s$ we have to apply exactly $s$ additional multiplications.
\end{proof}
The complexity of any multiplication in the realization of the multipliers can be estimated by $O(n^{1+\epsilon})$ operations in $A$, since the degree of any product of approximations in (\ref{Liste1}) and (\ref{Liste2}) is less than $n$, for $1\leq i\leq s$. As $s\leq n$, the complexity of the computation of $z_1,\dots,z_s$ is equal to $O(sn^{1+\epsilon})=O(n^{2+\epsilon})$ operations in $A$; that is, $O\left(n^{2+\epsilon}\delta^{1+\epsilon}\right)$ $\p$-small operations.

\noindent\textbf{\ref{en4} Basis multiplication}:

We determine the products $z_\kappa b_{\kappa,j}$, for $1\leq \kappa\leq s$ and $0\leq j< n_{\P_\kappa}$. Any $b_{\kappa,j}$ is given by $b_{\kappa,j}=g_{\kappa,j}(\t)$, where $g_{\kappa,j}(x)$ is a monic polynomial in $A[x]$ of degree $j<n_{\P_\kappa}$. For $1\leq \kappa\leq s$, the multiplier $z_\kappa$ is given by a polynomial in $A[x]$ of degree less than $n-n_{\P_\kappa}$ evaluated in $\t$. Hence, the computation of $z_\kappa b_{\kappa,j}$ can be realized at cost of $O(n^{1+\epsilon})$ operations in $A$. As $\sum_{i=1}^sn_{\P_i}=n$, we can compute all sets $z_\kappa \bb_\kappa$ at the cost of $O(n^{2+\epsilon})$ operations in $A$ which equates $O\left(n^{2+\epsilon}\delta^{1+\epsilon}\right)$ $\p$-small operations.

We summarize the results in the following lemma.

\begin{lemma}\label{compli}
Algorithm \ref{AlgopINT}  determines a $\p$-integral basis of $\oo$ by at most $$O\left(n^{1+\epsilon}\delta\log q+n^{1+\epsilon}\delta^{2+\epsilon}+n^{2+\epsilon}\delta^{1+\epsilon}\right)$$
$\p$-small operations.
\end{lemma}
\section{Experimental results}\label{Sec4}

We have implemented the $\p$-integral basis algorithm \ref{AlgopINT} from Section \ref{p_intBase} in \texttt{Magma} \cite{BCP} for an algebraic function field $F$ determined by a monic separable and irreducible polynomial $f\in A[x]$ with $A:=k[t]$ for a field $k$ as specified. Furthermore we have extended this algorithm to an integral basis algorithm; that is, we determine an integral basis of $\oo_F:=\mathrm{Cl}(A,F)$ by computing a triangular $\p$-integral basis of $\oo_\p$ for all $\p\in \mathrm{Spec}(A)$ with $v_\p(\dsc f)>1$ and merge those by an application of the CRT.

  We will compare the runtime of our algorithm with the implementation of the variant of the Round $2$ algorithm \cite{Round2} and with the implementation of the Quotient method presented in \cite{HN} both in \texttt{Magma}. 
All timings are in seconds and taken on a Linux server, with an Intel Xeon processor, running at 2.27 GHz, with 12 GB of RAM memory. For the first examples we use families of global function fields, which cover all the computational difficulties of the Montes algorithm \cite{GNP}. Later, we use function fields over the rationals. Note that we assume that all prime polynomials in $A$, which are divisible  by the ramified prime ideals of $\oo_F$, are precomputed.

\subsection{Global function fields}
At first we consider global function fields; that is, $f(t,x)\in \F_{q}[t,x]$ is defined over a finite field $\F_q$ with $q$ elements.

\subsubsection{Example 1}\label{ex1}
Let $f(t,x)=((x^6+4p(t)x^3+3p(t)^2x^2+4p(t)^2)^2+p(t)^6)^3+p(t)^k\in \mathbb{F}_{7}[t,x]$ with $p(t) = t^3+2$ and $1\leq k\leq 500$ and denote by $F$ the induced function field. Those prime ideals of $\oo_F$ which divide $p$ are among the ramified prime ideals of $\oo_F$. For $k\geq 17$ the ideal $p\cdot \oo_F$ splits into 6 prime ideals.
\medskip\\
\includegraphics[width=12.5cm]{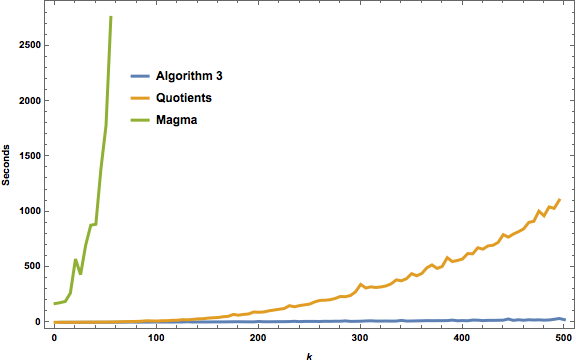}

\subsubsection{Example 2}
We consider the function field from Example \ref{ex1} for prime polynomials $p\in A$ with $1\leq \deg(p)\leq 220$ for $k=23$.
\medskip\\
\includegraphics[width=12.5cm]{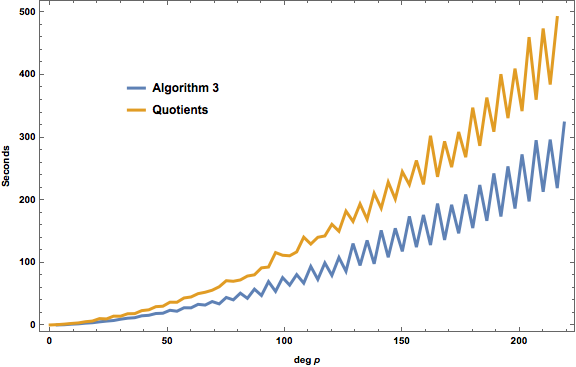}

\subsubsection{Example 3}

For $1\leq l\leq 7$ and $p := t^2+4$ we take the family of polynomials $f_l\in \F_{7}[t,x]$ as defined below and denote by $F_l$ the induced function fields. For $l>1$ we have $p\cdot \oo_{F_l}=\P^{\deg f_l}$.\\

\begin{tabular}{|l||c|c||c|c|}\hline $f_l$&Multiplier & Quotients &  \texttt{Magma} \\\hline 
$f_{ 1}(x)=x^2+p$ &0 &0 & 0 \\\hline
$f_{ 2}(x)=f_{1}(x)^{2}+(p-1)p^{3}x$ & 0& 0 &0 \\\hline
$f_{ 3}(x)=f_{2}(x)^{3}+p^{11}$  & 0& 0 &1 \\\hline
$f_{ 4}(x)=f_{3}(x)^{3}+p^{29}x f_{2}(x)$ & 0&   0&520 \\\hline
$f_{ 5}(x)=f_{4}(x)^{2}+(p-1)p^{42}xf_{1}(x)f_{3}(x)^2$   & 1& 8&60000 \\\hline
$f_{ 6}(x)=f_{5}(x)^{2}+p^{88}xf_{3}(x)f_{4}(x)$  & 10&20  &$-$  \\\hline
$f_{ 7}(x)=f_{6}(x)^{3}+p^{295}xf_{2}(x)f_{4}(x)f_{5}(x)$&  $226$ &$21530$ &$-$  \\\hline

 \end{tabular}

\subsubsection{Example 4}

We consider $f(t,x) := x^{4330} -(t^2+1)(x^2-1)- (t^8+2t^6 +1)x$ in $\F_3[t,x]$ and denote by $F/\F_3$ the induced function field. It holds $t\cdot \oo_{F}=\prod_{i=1}^{14}\P_i$.\\

\begin{minipage}{\linewidth}

\renewcommand{\footnoterule}{}
\renewcommand{\thefootnote}{\alph{footnote}}

\begin{tabular}{|c|c||c|c|}\hline Multiplier & Quotients &  \texttt{Magma} \\\hline 
 68 &486 &$-$\footnotemark[1]\  \\\hline
 \end{tabular}
\footnotetext[1]{All virtual memory has been exhausted, so  \texttt{Magma} cannot perform this statement.}

\end{minipage}

\subsection{Function fields over $\Q$}
The following examples are taken from \cite{Decker}. For $1\leq l\leq 6$ we take the polynomials $g_l\in \Q[t,x]$ as below and denote by $F_l$ the induced function fields. \\


\resizebox{0.83\textwidth}{!}{\begin{minipage}{\linewidth}

\renewcommand{\footnoterule}{}
\renewcommand{\thefootnote}{\alph{footnote}}

\begin{tabular}{|c|c||c|c||c|}\hline $l$&$g_l$&Ramification \\\hline 
$1$& $(x^4+2t^3x^2+t^6+t^5x)^3+(tx)^{11}$ &$t\cdot \oo_{F_1}=\P^{12}$ \\\hline
$2$& $x^{22}+t^{22}+z^{22}+2((tz)^{11}-(tx)^{11}+(xz)^{11}),\ z=t-2x+1$ &$t\cdot \oo_{F_2}=\P_1^{2}\P_2^{2}$\ \footnotemark[2] \\\hline
$3$& $ (x^{15}+2t^{38})(x^{19}+7t^{52})+x^{100}$ &$t\cdot \oo_{F_3}=\P_1^{19}\P_2^{15}\P_3\cdots \P_6$ \\\hline
$4$& $x^{200}+tx^{13}+t^4x^5+t^5+2t^4+t^3$ &$t\cdot \oo_{F_4}=\P_1^{13}\P_2^{187}$ \\\hline
$5$& $ x^{401} + t^{500} + t^2$ &$t\cdot \oo_{F_5}=\P^{401}$ \\\hline
$6$& $ x^{500}+tx^2+t^{400}$ &$t\cdot \oo_{F_6}=\P_1^{2}\P_2^{498}$ \\\hline

\end{tabular}\\\medskip

\footnotetext[2]{There are more prime ideals of $\oo_{F_2}$, which are ramified and do not divide $t$.}

\end{minipage}}
\\\medskip

\begin{tabular}{|c|c||c|c||c|c|}\hline $l$&$\deg g_l$&Multiplier & Quotients &  \texttt{Magma} \\\hline 
$1$& $12$ & 0&  0 & 2 \\\hline
$2$& $22$  &35 &53  &183 \\\hline
$3$& $100$   & 0&0 &$-$ \\\hline
$4$& $200$  & 1& 1 &$-$  \\\hline
$5$& $401$ & 0&1 & $-$ \\\hline
$6$& $500$  &1 &5  &$-$ \\\hline

 \end{tabular}


\begin{thebibliography}{}
\bibitem{BNS}
J.-D. Bauch, E. Nart, H. D. Stainsby, \emph{Complexity of OM factorizations of polynomials over local fields},	LMS J. Comput. Math. {\bf16} (2013), 139--171
\bibitem{Icke}
J.-D. Bauch, \emph{Lattices over polynomial Rings and Applications to Function Fields}, Ph.D. thesis, Universidad Aut\'{o}noma de Barcelona (2014).


\bibitem{BCP}
W. Bosma, J. Cannon, C. Playoust, \emph{The Magma algebra system I: The user language},  J. Symbolic Computation, 24 3/4:235265, 1997.


\bibitem{Decker}
J. Boehm, W. Decker, S. Laplagne, G. Pfister, \emph{Computing integral bases via localization and Hensel lifting},  arXiv:1505.05054v1 [math.NT], 2015.

\bibitem{Round2}
D. Ford, P. Letard, \emph{Implementing the Round Four maximal order algorithm}, J. de Th\'eorie des Nombres de Bordeaux, {\bf 6} (1994), no. 1, 39--80.




\bibitem{okutsu}
J. Gu\`{a}rdia, J.  Montes, E.  Nart, \emph{Okutsu invariants and Newton polygons}, Acta Arith. {\bf145} (2010), 83--108.


\bibitem{HN}
J. Gu\`{a}rdia, J.  Montes, E.  Nart, \emph{Higher Newton polygons and integral bases}, J. Number Theory {\bf 147} (2015), 549--589.

\bibitem{HN2}
J. Gu\`{a}rdia, J.  Montes, E.  Nart, \emph{Higher  Newton polygons in the computation of discriminants and prime ideal decomposition in number fields}, J. Th\'eor. Nombres Bordeaux {\bf 23} (2011), no. 3, 667--696.

\bibitem{HN1}
J. Gu\`{a}rdia, J.  Montes, E. Nart, \emph{Newton polygons of higher order in algebraic number theory}, Trans. Amer. Math. Soc.  {\bf 364} (2012), no. 1, 361--416.

\bibitem{NewC}
J. Gu\`{a}rdia, J. Montes, E.  Nart, \emph{A new computational approach to ideal theory in number fields}, Found. Comput. Math. {\bf 13} (2013), 729--762.

\bibitem{GNP}
J. Gu\`{a}rdia,  E. Nart, S. Pauli, \emph{Single-factor lifting and factorization of polynomials over local fields}, J. Symb. Comput. {\bf 47} (2012), 1318--1346.


\bibitem{OIN}
J. Gu\`{a}rdia, J. Montes, E.  Nart, \emph{Genetics of polynomials over local fields},  Proceedings of AGCT14, Contemporary Mathematics 637 (2015), 207-241.

\bibitem{Hensel}
K. Hensel, \emph{Theorie der algebraischen Zahlen}, Teubner, Leipzig, Berlin, 1908.

\bibitem{NK}
J. Neukirch, \emph{Algebraische Zahlentheorie}, Springer Verlag, 1991.

\bibitem{Oku}
K. Okutsu, \emph{Construction of Integral Basis. I}, Proceedings of the Japan Academy, 58, Ser. A (1982), 47-49, 87-89.

\bibitem{Sch}
W. M. Schmidt, \emph{Construction and estimation of bases in function fields}, J. Number Theory {\bf 39} (1991), 181--224. 

\bibitem{Gath2}
A. Sch\"onhage and V. Strassen, ÔSchnelle Multiplikation gro{\ss}er ZahlenÕ, Computing 7 (1971) 281-292.


\bibitem{JPS}
J.-P. Serre, \emph{Corps locaux}, 4th corrected Edition, Hermann, Paris, 2004.


\end{thebibliography}
\end{document}